\documentclass[12pt,leqno]{article}
\usepackage{amsmath,amssymb,amsbsy,amsfonts,amsthm,latexsym,
         amsopn,amstext,amsxtra,euscript,amscd,amsthm,mathdots,hyperref}

\allowdisplaybreaks

\numberwithin{equation}{section}

\newtheorem{lem}{Lemma}

\newtheorem{prop}{Proposition}

\newtheorem{thm}{Theorem}

\newtheorem{rem}{Remark}

\newcommand{\seqnum}[1]{\href{https://oeis.org/#1}{\rm \underline{#1}}}

\begin{document}

\begin{large}
\centerline{\Large \bf The Integer Values of $\sum_{i=1}^{n}\frac{1}{\gcd(n,i)}$}
\end{large}
\vskip 10pt
\begin{large}
\centerline{\sc  Jean-Marie De Koninck}
\end{large}
\vskip 10pt
\begin{large}
\centerline{\sc  Patrick Letendre}
\end{large}
\vskip 10pt
\begin{abstract}
We study the set of integers $n \le x$ for which $\sum_{i=1}^{n}\frac{1}{\gcd(n,i)}$ takes an integer value. We show that the number of such integers is $\ll x^{2/3+\epsilon}$ for all fixed $\epsilon > 0$.
\end{abstract}
\vskip 10pt
\noindent MSC2020: 11N37, 11N64.

\noindent Keywords: arithmetic function, integer values.

\section{Introduction and notation}

For each integer $n \ge 1$, we consider the arithmetic function
\[
\mathcal{H}(n):=\sum_{i=1}^{n}\frac1{\gcd(n,i)} = \frac{h(n)}{n},
\]
where
\[
h(n) := \prod_{p^{\alpha} \parallel n}\frac{p^{2\alpha+1}+1}{p+1}.
\]
Clearly, $h(n) \le n^{2}$ for $n \ge 1$. Also, the values of $h$ at primes play a special role, so it is useful to define the polynomial
\[
g(z) := z^{2}-z+1,
\]
for which $h(p)=g(p)$ for every prime $p$.

Since $\mathcal{H}(n)$ is always rational, it is natural to study the integers $n$ for which $\mathcal{H}(n)$ is an integer. We denote
\[
\mathcal{S} := \{n\in\mathbb{N}:\ \mathcal{H}(n)\in\mathbb{N}\}, \qquad S(x) := |\mathcal{S} \cap [1,x]| \quad (x \ge 1).
\]

It is easy to verify that every integer $n > 1$ belonging to $\mathcal{S}$ satisfies the following elementary properties:
\begin{enumerate}
  \item $n$ is odd;
  \item $\omega(n) \ge 2$, where $\omega(n)$ is the number of distinct prime factors of $n$;
  \item $P(n) < n^{2/3}$, where $P(n)$ denotes the largest prime factor of $n$.
\end{enumerate}
The first elements of $\mathcal{S}$ are
\begin{align*}
& 1,\ 614341,\ 618233,\ 1854699,\ 11746427,\ 26584019,\ 35239281,\ 79752057,\\
& 85393399,\ 118082503,\ 345592247,\ 354247509,\ 505096361,\ 802597537,\ \dots,
\end{align*}
see \cite{jmdk} and \seqnum{A317480}.

Let $\sigma(n)$ denote the sum of the divisors of $n$. In \cite{ew}, Wirsing showed that, uniformly for all rational numbers $\frac{a}{b}$, the number of integers $n \le x$ satisfying the equation
\[
\frac{\sigma(n)}{n} = \frac{a}{b}
\]
is $\ll \exp\bigl(c_{1} \frac{\log x}{\log \log x}\bigr)$, for some constant $c_{1} > 0$ and all $x \ge 10$. His argument also implies that the function $\frac{\sigma(n)}{n}$ very rarely takes integer values, since the number of possible integer values is only $O(\log\log x)$ by the well-known inequality
\[
\frac{\sigma(n)}{n} \ll \log\log x
\]
that holds uniformly for $n \le x$ when $x \ge 10$.

In our setting, the situation is more delicate, since $\mathcal{H}(n)$ is essentially of the same order of magnitude as $n$. Moreover, it is not clear how one could obtain an upper bound for the number of integers $n \le x$ satisfying $\mathcal{H}(n) = m$ for a fixed integer $m$. We therefore employ a different approach to prove the following theorem.

\begin{thm}\label{thm:1}
For all $x \ge 100$, we have
\[
S(x) \ll x^{2/3}\exp\Bigl(\frac{1}{6}\bigl(1+o(1)\bigr)\frac{\log x\log\log\log x}{\log\log x}\Bigr).
\]
\end{thm}

Finally, we note that the set of integers $n$ such that $n$ divides $\sigma(n^{2})$ has similar properties. The first terms are (\seqnum{A232354})
\begin{align*}
&1,\ 39,\ 793,\ 2379,\ 7137,\ 13167,\ 76921,\ 78507,\ 230763,\ 238887,\ 549549,\\
&692289,\ 863577,\ 1491633,\ 1672209,\ 2076867,\ 4317885,\ 7615179,\ 8329831,\ \dots
\end{align*}
and a corresponding theorem holds.

Throughout the paper, we use constants $c_{1},c_{2},\ldots$ and $C_{1},C_{2},\ldots$ whose values may vary from one occurrence to another. This should not cause any confusion to the reader.

\section{Beginning of the proof of Theorem \ref{thm:1}}\label{proof-intro}

For each $n \in \mathcal{S}$, we consider the factorization $n = ab$, where $a$ is cubefull, $\gcd(a,b) = 1$, and
\[
b = p_{1}^{\alpha_{1}} \cdots p_{\ell}^{\alpha_{\ell}},
\]
where $p_{1} < \cdots < p_{\ell}$ and $\alpha_{i} \le 2$ for $i = 1,\dots,\ell$. Since $n \in \mathcal{S}$, for each $i = 1,\dots,\ell$, we have that $p_{i}$ divides either $h(a)$ or $h(p_{j}^{\alpha_{j}})$ for some $j = 1,\dots,\ell$ with $i \neq j$. If the latter case occurs, we retain only the divisibility relation involving the largest corresponding prime, say $p_{j_{i}}$, and record the pair $(p_{i}, p_{j_{i}})$. Finally, if only the first case occurs, we record the pair $(p_{i}, a)$. We also consider the pair $(a,a)$. Let $\mathcal{C}$ denote the set of these $\ell+1$ pairs.

We associate each prime $p_{i}$ with a vertex $v_{i}$, and likewise associate $a$ with the vertex $v$, and denote this set of $\ell+1$ vertices by $V$. Let $E$ be the set of $\ell+1$ directed pairs obtained from $\mathcal{C}$ via this association, where $(u_{1}, u_{2})$ denotes an edge from $u_{1}$ to $u_{2}$. In this way, we construct a directed graph $G(V,E) = G(n)$.

Let $G_{0}(V_{0}, E_{0}), G_{1}(V_{1}, E_{1}), \dots, G_{t}(V_{t}, E_{t})$ be the list of connected subgraphs of $G$. Exactly one of these graphs contains the vertex $v$, and we will assume it is $G_{0}$.

The graphs $G_{i}$, $i=0,1,\dots,t$, are of a particular type. They are connected {\it functional graphs} since they represent the action of a function $f_{i}(\cdot)$. Lemma \ref{lem:1} tells us that such graphs have a unique cycle and that each trajectory follows a tree that leads to this cycle.

We thus factorize
$$
n = n_{1} n_{2} n_{3}
$$
where $n_{1}$ is the part of $n$ associated with $G_{0}$; it includes the entire cubefull part (which corresponds to the 1-cycle part of $G_{0}$) and the product, with the same powers as in $n$, of the primes associated to $V_{0}$. Similarly, $n_2$ is the product, over all $G_{i}$ whose cycles have length $2$, of the primes associated with these components, with the same powers as in $n$. The integer $n_3$ is defined similarly, using the $G_{i}$ whose cycles have length at least $3$. Clearly, $n_{1}$, $n_{2}$, and $n_{3}$ thus defined are mutually coprime and defined without ambiguity. Also, taking each $n_{i}$ alone recovers the corresponding graphs.

Let $\mathcal{W}_{i}$, for $i=1,2,3$, be the set of integers associated respectively with a 1-cycle graph, 2-cycle graphs, and 3-cycles or longer. We also denote
$$
W_{i}(X) = |\mathcal{W}_{i} \cap [1,X]| \quad (X \ge 1).
$$
In Sections \ref{W_{3}(X)}, \ref{W_{2}(X)}, and \ref{W_{1}(X)} below, we will show that
$$
W_{i}(X) \ll X^{2/3}\exp\Bigl(\frac{1}{6}\bigl(1+o(1)\bigr)\frac{\log X\log\log\log X}{\log\log X}\Bigr) \quad (i=1,2,3)
$$
for $X \ge 100$. We can then write
\begin{eqnarray*}
S(x) & = & \sum_{\substack{n_{1}n_{2}n_{3} \le x \\ n_{1}n_{2}n_{3} \in \mathcal{S} \\ \gcd(n_{i},n_{j})=1}}1\\
& \le & \sum_{\substack{i_{1},i_{2},i_{3} \ge 0 \\ i_{1}+i_{2}+i_{3} \le \frac{\log x}{\log 2}+3}}W_{1}(2^{i_{1}})W_{2}(2^{i_{2}})W_{3}(2^{i_{3}})\\
& \ll & x^{2/3}\exp\Bigl(\frac{1}{6}\bigl(1+o(1)\bigr)\frac{\log x\log\log\log x}{\log\log x}\Bigr)
\end{eqnarray*}
from which the result follows.

It only remains to establish the upper bound for $W_{i}(X)$ ($i=1,2,3$).

\section{Some important lemmas}

\begin{lem}\label{lem:1}
Let $G = (V, E)$ be a finite directed graph such that
\begin{itemize}
    \item $V$ is a finite set of cardinality $n \ge 1$;
    \item there exists a function $f : V \to V$ such that $E = \{ (v, f(v)) :\ v \in V \}$;
    \item the undirected graph obtained by replacing each directed edge $(u, v) \in E$ with the undirected edge $\{u, v\}$ is connected.
\end{itemize}

Then
\begin{enumerate}
    \item for every vertex $v \in V$, the sequence $v, f(v), f(f(v)), \dots$ enters a cycle in at most $n$ steps;
    \item the graph $G$ contains exactly one directed cycle;
    \item each vertex on the cycle is the root of a $($possibly trivial$)$ directed tree whose edges point toward the cycle.
\end{enumerate}
\end{lem}

\begin{proof}
This is an elementary result about {\it functional graphs}. Point 1 is immediate, and point 3 follows directly from point 2. We may assume that $n \ge 2$, since otherwise the result is trivial.

To prove point 2, suppose for contradiction that $G$ contains at least two disjoint directed cycles, denoted $C_{1}$ and $C_{2}$. Since the underlying undirected graph is connected, there exists an undirected path connecting a vertex $u_{1} \in C_{1}$ to a vertex $u_{2} \in C_{2}$. Let
\[
u_{1} = v_{0},\ v_{1},\ \dots,\ v_{k} = u_{2}
\]
be such a path, where each pair $\{v_{i}, v_{i+1}\}$ corresponds to an edge in the underlying undirected graph.

Let $i$ be the largest index such that $v_{i} \in C_{1}$, and set $u = v_{i} \in C_{1}$, $w = v_{i+1} \notin C_{1}$. Thus, $w$ is connected to $C_{1}$ by an edge. Since the outgoing edge from $u$ belongs to $C_{1}$ (as $u$ lies on a directed cycle), the edge joining $u$ and $w$ cannot be directed from $u$ to $w$. Hence it must be directed from $w$ to $u$, that is, $f(w) = u$. Therefore, the trajectory of $w$ under $f$ enters $C_{1}$.

Applying the same argument successively to $v_{i+1}, v_{i+2}, \dots, v_{k}=u_{2}$, we conclude that the trajectory of $u_{2}$ under $f$ enters $C_{1}$, which contradicts the fact that $u_{2}$ belongs to the disjoint cycle $C_{2}$.

This contradiction proves point 2, and hence also point 3.
\end{proof}

A {\it Dyck word} of length $2n$ is a string consisting of $n$ opening parentheses ``('' and $n$ closing parentheses ``)'', arranged in such a way that:
\begin{enumerate}
    \item At no point in the string do the closing parentheses outnumber the opening parentheses when reading from left to right.
    \item The string contains equally many opening and closing parentheses.
\end{enumerate}
Equivalently, a Dyck word is a well-formed parenthesis string in which each closing parenthesis matches a previously unmatched opening parenthesis.

\begin{lem}\label{lem:2}
The number of Dyck words of length $2n$ is given by the $n$-th Catalan number
\[
C_{n} = \frac{1}{n+1}\binom{2n}{n}.
\]
\end{lem}

\begin{proof}
Let $D_{n}$ denote the set of Dyck words of length $2n$. There is only one Dyck word of length $0$ (the empty word), so $|D_{0}|=1$. Every non-empty Dyck word admits a unique decomposition of the form $(A)B$, where
\begin{itemize}
    \item $A \in D_{k}$ is the shortest Dyck word such that $(A)$ is an initial segment of the original word, for some $0 \le k < n$;
    \item $B \in D_{n-1-k}$ is a Dyck word of length $2(n-1-k)$.
\end{itemize}

This yields the recurrence
\[
|D_{n}| = \sum_{k=0}^{n-1} |D_{k}| \cdot |D_{n - 1 - k}| \qquad (n \ge 1).
\]
This is precisely the standard recurrence relation for the Catalan numbers. Since $C_{0} = |D_{0}| = 1$ and both sequences satisfy the same recurrence, the result follows by induction.
\end{proof}

The next lemma is a variant of the upper bound in Théorème 2 of \cite{pe:jln}.

\begin{lem}\label{lem:3}
Let $0 < \delta < 1$ be a fixed real number. Uniformly for $x \ge 100$ and $\delta \le \theta \le 1$, we have
\begin{align*}
V_{\theta}(x) & := |\{n \le x:\ \omega(n) > \frac{\theta\log x}{\log\log x}\}| \\
& \ll x^{1-\theta}\exp\Bigl(2\theta\bigl(1+o_{\delta}(1)\bigr)\frac{\log x\log\log\log x}{\log\log x}\Bigr).
\end{align*}
\end{lem}

\begin{proof}
We may assume that $x$ is sufficiently large. We write $q$ for a prime power. Set $k := \bigl\lceil\frac{\theta\log x}{\log\log x}\bigr\rceil$. Then
\begin{align*}
V_{\theta}(x) & = \sum_{\substack{n \le x \\ \omega(n) \ge k}}1 \le \sum_{\substack{q \mid n \Rightarrow q \le x \\ \omega(n) \ge k}}\frac{x}{n} \le \sum_{j \ge k}\frac{x}{j!}\Bigl(\sum_{q \le x}\frac{1}{q}\Bigr)^{j} \le \sum_{j \ge k}\frac{x}{j!}\bigl(\log\log x+c\bigr)^{j} \\
& \ll \frac{x}{k!}\bigl(\log\log x+c\bigr)^{k} \ll x\Bigl(\frac{e\log\log x+ec}{k}\Bigr)^{k} \\
& \ll x^{1-\theta}\exp\Bigl(2\theta\bigl(1+o_{\delta}(1)\bigr)\frac{\log x\log\log\log x}{\log\log x}\Bigr).
\end{align*}
Here we used Mertens's theorem (see \cite[p.~33]{gt}) together with Stirling's formula. The proof is complete.
\end{proof}

\begin{lem}\label{lem:4}
Let $X \ge 3$. There is a constant $c \ge 1$ such that every integer $n \le X$ satisfies
\[
\omega(n) \le \frac{c\log X}{\log\log X}.
\]
\end{lem}

\begin{proof}
This is a well-known consequence of the prime number theorem.
\end{proof}

\begin{lem}\label{lem:5}
Let $|A|, |B|, |C|, |D|, |E|, |F| \le M^{O(1)}$ and assume that $\Delta = B^{2}-4AC$ is not a perfect square (in particular, $\Delta \neq 0$). Then the diophantine equation
\[
Ax^{2}+Bxy+Cy^{2}+Dx+Ey+F = 0
\]
has at most $M^{o(1)}$ solutions in integers $x, y$ with $1 \le |x|, |y| \le M^{O(1)}$.
\end{lem}

\begin{proof}
This is Proposition 1 of \cite{jc:mzg}.
\end{proof}

\begin{rem}
By carefully inspecting the proof of the previous lemma (see \cite{jc:mzg}), one obtains the following more quantitative version. Let $c_{1}$ and $c_{2}$ be two fixed positive constants. Under the same assumption on $\Delta$, if $|A|,|B|,|C|,|D|,|E|,|F| \le M^{c_1}$, then the number of integer solutions satisfying $1 \le |x|, |y| \le M^{c_2}$ is at most
\[
C_1(c_1,c_2)
\exp\left(C_2(c_1,c_2)\frac{\log M}{\log\log M}\right) \qquad (M \ge 10),
\]
where the constants are effectively computable.
\end{rem}

\section{Upper bound for $W_{3}(X)$}\label{W_{3}(X)}

We may assume that $X$ is sufficiently large so that all the expressions and inequalities used throughout the proof are valid. Also, observe that $1 \in \mathcal{W}_{3}$. Let $1 < d \le X$ be an integer in $\mathcal{W}_{3}$ satisfying $\omega(d) = k$. We construct the directed graph $G(V, E)=G(d)$ as explained in Section \ref{proof-intro}. $G$ is a functional graph, but not necessarily connected. It is a collection of connected functional graphs $G_{1},\dots,G_{t}$, each of which satisfies the conditions of Lemma \ref{lem:1}.

Let us isolate one graph $G_{i}(V_{i}, E_{i})$ and imagine it drawn so that the cycle turns clockwise, and the trees are clearly separated from each other, placed near their root, on the outside of the cycle. Let $N_{i}$ be the product of the prime numbers associated with the vertices of $V_{i}$. Since the cycle in $G_{i}$ has length at least 3, there exists a vertex in this cycle associated with a prime $p_{i}$ satisfying $p_{i} < N_{i}^{1/3}$. From this vertex $v_{i}$, corresponding to the smallest $p_{i}$ in the cycle, we define a new graph $G'_{i} (=G'_{i}(V_{i},E'_{i}))$. We obtain $G'_{i}$ by deleting the outgoing edge of $v_{i}$, which necessarily belongs to the cycle. Then, by unfolding $G'_{i}$ so that $v_{i}$ is placed at the bottom and what was the interior of the cycle is now oriented to the left and upward, we obtain a rooted oriented tree with $v_{i}$ as its root.

Let us proceed similarly for each of the graphs $G_{1}(V_{1},E_{1}),\dots,G_{t}(V_{t},E_{t})$, and suppose that the trees $G'_{1}(V_{1},E'_{1}),\dots,G'_{t}(V_{t},E'_{t})$ are drawn side by side, from left to right. We then denote by $G'=G'(V,E')=G'(V,\cup_{i=1}^{t}E'_{i})$ the graph formed by the union of these trees, that is, a forest. We also define the integer $U(d):=p_{1}\cdots p_{t}$. Clearly, $U(d) < d^{1/3} \le X^{1/3}$.

Up to this point, starting from an integer $d \in \mathcal{W}_{3} \cap [1,X]$ such that $\omega(d)=k$, we have constructed a pair $(G'(V,E'),u)$ where $\omega(u)$ corresponds to the number of connected components of $G'$, with $u < X^{1/3}$ and $|V|=k$. Our result is obtained by reasoning in the reverse direction.

First of all, there is a well-known correspondence between Dyck words and forests. In particular, we deduce from Lemma \ref{lem:2} that there are at most $C_{k} \le \frac{4^{k}}{k+1}$ forests $G'(V,E')$ with $|V|=k$ and $t$ connected components.

Next, using the fact that $u < X^{1/3}$ and that we have two choices of exponents for each vertex, we deduce that the number of $t$-tuples $(p_{1}, \dots, p_{t})$ that can be associated with the $t$ roots of $G'$ is bounded by
\begin{eqnarray}\nonumber
t!|\{u \le X^{1/3}:\ \omega(u)\mu(u)^{2} = t\}| & \ll & t!2^{t}X^{1/3}\\ \label{W3:1}
& \ll & \Bigl(\frac{2c \log X}{\log \log X}\Bigr)^{t}X^{1/3}
\end{eqnarray}
for some constant $c$. We used Lemma \ref{lem:4}.

Finally, given $G'$ together with the $t$-tuple $(p_{1},\dots,p_{t})$, one can essentially reconstruct the complete list of prime factors of $d$. Indeed, in $G'_{1}$, the vertices connected to the root are chosen among the $\omega(h(p_{1}))+\omega(h(p_{1}^{2}))$ prime factors of $h(p_{1})$ and $h(p_{1}^{2})$. Then, the vertices connected to these vertices are in turn chosen among the $\omega(h(p))+\omega(h(p^{2}))$ for the corresponding primes $p$, and so on, following the tree structure. Since we have uniformly
$$
\omega(h(p))+\omega(h(p^{2})) \le \frac{2c\log X}{\log \log X} \quad (p \le X),
$$
possibly adjusting the value of the constant $c$, we deduce that the pair $(G'_{1}, p_{1})$ can generate at most
$$
\Bigl(\frac{4c\log X}{\log \log X}\Bigr)^{|E'_{1}|} = \Bigl(\frac{4c\log X}{\log \log X}\Bigr)^{|E_{1}|-1}
$$
possible associations of prime numbers to the vertices of $G'_{1}$. Reasoning similarly for each connected component, we conclude that the pair $(G',(p_{1},\dots,p_{t}))$ yields at most
$$
\Bigl(\frac{4c\log X}{\log \log X}\Bigr)^{k-t}
$$
distinct associations of primes with the vertices of $G'$, that is, at most as many different values for $d$ given the roots.

Putting all this together, and using the fact that $1 \le t \le k/3$, we deduce that the number of integers $d \le X$ belonging to $\mathcal{W}_{3}$ and satisfying $\omega(d) = k$ is bounded by
$$
4^{k}\Bigl(\frac{4c\log X}{\log \log X}\Bigr)^{k}X^{1/3}.
$$

To complete the proof, fix $\theta > 0$ to be chosen later. We have
\begin{eqnarray*}
W_{3}(X) & = & \sum_{\substack{d \le X \\ d \in \mathcal{W}_{3} \\ \omega(d) \le \frac{\theta\log X}{\log\log X}}}1+\sum_{\substack{d \le X \\ d \in \mathcal{W}_{3} \\ \omega(d) > \frac{\theta\log X}{\log\log X}}}1 \ll \sum_{k \le \frac{\theta\log X}{\log\log X}}\Bigl(\frac{16c\log X}{\log \log X}\Bigr)^{k}X^{1/3}+V_{\theta}(X)\\
& \ll & \Bigl(\frac{16c\log X}{\log \log X}\Bigr)^{\frac{\theta\log X}{\log\log X}}X^{1/3}+X^{1-\theta}\exp\Bigl(2\theta\bigl(1+o(1)\bigr)\frac{\log X\log\log\log X}{\log\log X}\Bigr)\\
& \ll & X^{2/3}\exp\Bigl(\frac{1}{6}\bigl(1+o(1)\bigr)\frac{\log X\log\log\log X}{\log\log X}\Bigr)
\end{eqnarray*}
by choosing $\theta := \frac{1}{3}+\frac{1}{2}\frac{\log\log\log X}{\log\log X}$. We used Lemma \ref{lem:3} on the second line.

\section{Upper bound for $W_{2}(X)$}\label{W_{2}(X)}

The analysis of $\mathcal{W}_{2}$ is based on the same idea as that of $\mathcal{W}_{3}$, but relies on a more careful analysis of the cycles. Let $d \in \mathcal{W}_{2} \cap [1,X]$ satisfy $\omega(d) = k \ge 2t$. We distinguish between two cases. In the first case, at least one of the first two primes $p$ associated with the vertices in the cycle satisfies $p^{2} \| d$. The second case is when both primes $p$ and $q$ satisfy $p \| d$ and $q \| d$. Let $G_{1}, \dots, G_{t}$ be the graphs derived from $d$, and suppose that $G_{1}, \dots, G_{s}$ (with $s = 0, \dots, t$) fall into the first case, and $G_{s+1}, \dots, G_{t}$ into the second.

Let $d_{i}$ be the divisor of $d$ associated with the graph $G_{i}$, that is, $p \mid \gcd(d,N_{i})$ and if $p^{\alpha} \| d$ then $p^{\alpha} \| d_{i}$. In $G_{i}$, if we are in the first case, the cycle contains either $(p_{i},q_{i}^{2})$ or $(p_{i}^{2},q_{i}^{2})$. Clearly, we then have $\min(p_{i},q_{i}) < d_{i}^{1/3}$. The following proposition provides the analogous tool for handling the second case.

\begin{prop}\label{prop:1}
Let $z \ge 1$ be a fixed real number. Consider the system
$$
\left\{\begin{array}{l} 
ap = q^{2}-q+1 \\ 
bq = p^{2}-p+1 
\end{array}\right.
$$
where $a,b \ge 1$ are arbitrary integers and $p,q$ are distinct odd prime numbers satisfying $pq \le z$. Then the number of solutions to this system is at most $c_{2}\frac{z^{1/3}}{(\log 2z)^{2}}$, where $c_{2}$ is an explicitly computable constant.
\end{prop}

\begin{proof}
We may assume that $z$ is sufficiently large. Let $p$ and $q$ be distinct odd prime numbers, and let $a$ and $b$ be positive integers such that
\begin{equation}\label{prop:1:1}
ap = g(q) \quad \mbox{and} \quad bq = g(p).
\end{equation}
It follows that
\begin{eqnarray*}
a^{2}bq & = & (ap)^{2}-a^{2}p+a^{2}\\
& = & g(q)^{2}-ag(q)+a^{2},
\end{eqnarray*}
from which $q \mid g(a)$. Since $g(a) > 0$, we deduce that $q < a^{2}$, and therefore
$$
p\sqrt{q} < ap = g(q) < q^{2},
$$
which implies $p < q^{3/2}$. Without loss of generality, we have $q < p < q^{3/2}$, and hence $pq < q^{5/2}$. We deduce that $\min(p,q) > (pq)^{2/5}$ and $\max(p,q) < (pq)^{3/5}$. Consequently, we obtain
\begin{equation}\label{prop:1:2}
\frac{p}{q}+\frac{q}{p} < (pq)^{1/5}+1 \le z^{1/5}+1.
\end{equation}
Now, from \eqref{prop:1:1} we deduce that $(x,y) = (p,q)$ is a solution to the diophantine equation
$$
Kxy=x^{2}-x+1+y^{2}-y
$$
for some odd integer $K \ge 1$. Indeed, the right-hand side is a multiple of $pq$. From \eqref{prop:1:2}, we also observe that $K < z^{1/5}+1$. Since $K^{2}-4$ is never a square, Lemma \ref{lem:5} implies that the total number of such solutions $(p,q)$ is $\ll z^{1/5+o(1)} < c_{2}z^{1/3}/(\log 2z)^{2}$, as claimed.
\end{proof}

Suppose that $\prod_{j=1}^{s}d_{j}\le 2^{i}$ and that $d_{j} \le 2^{i_{j}}$ for each $j = s+1, \dots, t$. Reasoning as in \eqref{W3:1}, we obtain the analogous upper bound for the number of root assignments in the first case. Specifically, we use the estimate $\bigl(\frac{2c\log X}{\log\log X}\bigr)^{s} 2^{i/3}$. From Proposition \ref{prop:1}, we find that the number of root assignments is
\begin{eqnarray*}
& \le & \sum_{s=0}^{t}\sum_{\substack{i+i_{s+1}+\dots+i_{t} \\ \le \frac{\log X}{\log 2}+1+t-s}}\Bigl(\frac{2c\log X}{\log \log X}\Bigr)^{s}2^{i/3}\frac{c_{3}2^{i_{s+1}/3}}{(i_{s+1}+1)^{2}}\cdots\frac{c_{3}2^{i_{t}/3}}{(i_{t}+1)^{2}}\\
& \ll & X^{1/3}\sum_{s=0}^{t}\sum_{\substack{i+i_{s+1}+\dots+i_{t} \\ \le \frac{\log X}{\log 2}+1+t-s}}\Bigl(\frac{2c\log X}{\log \log X}\Bigr)^{s}\frac{(2^{1/3}c_{3})^{t-s}}{((i_{s+1}+1)\cdots(i_{t}+1))^{2}}\\
& \ll & X^{1/3}(\log X)\sum_{s=0}^{t}\Bigl(\frac{2c\log X}{\log \log X}\Bigr)^{s}(2^{1/3}(\pi^{2}/6)c_{3})^{t-s}\\
& \ll & X^{1/3}(\log X)\Bigl(\frac{c_{4}\log X}{\log \log X}\Bigr)^{t}
\end{eqnarray*}
for some constant $c_{4}$. This last inequality is the exact analogue of \eqref{W3:1}. From here, we proceed in almost exactly the same way to reach the desired conclusion, namely that
$$
W_{2}(X) \ll X^{2/3}\exp\Bigl(\frac{1}{6}\bigl(1+o(1)\bigr)\frac{\log X\log\log\log X}{\log\log X}\Bigr).
$$

\section{Upper bound for $W_{1}(X)$}\label{W_{1}(X)}

Again, the underlying idea of the proof remains the same. Let $d \in \mathcal{W}_{1} \cap [1,X]$ and suppose that $d = ab$ with $\gcd(a,b) = 1$, where $a$ is cubefull and $b$ is cubefree satisfying $\omega(b) = k$. Let $G_{0}$ be the directed graph associated with $d$ via the process described in Section \ref{proof-intro}. By construction, $G_{0}$ is connected and has a special vertex, say $v$, which is the 1-cycle that makes the graph functional. By Lemma \ref{lem:1}, $G_{0}$ consists of trees that flow into this unique cycle. We define $G'_{0}$ by removing the edge $(v,v)$. Thus, $G'_{0}$ is a rooted directed tree with $v$ as its root.

Once again, given $G'_{0}$ and $a$, we have at most $2\omega(h(a)) \le \frac{2c\log X}{\log\log X}$ choices for each vertex connected to $v$. Moreover, this bound $\frac{2c\log X}{\log\log X}$ is uniform across all other vertices as well. Since there are $\ll X^{1/3}$ possible choices for $a$ and at most $4^{k+1}$ possible choices for $G'_{0}$, we obtain at most
$$
\ll \Bigl(\frac{8c\log X}{\log\log X}\Bigr)^{k}X^{1/3}
$$
choices for such $d$. From this, the final calculation in Section \ref{W_{3}(X)} gives
$$
W_{1}(X) \ll X^{2/3}\exp\Bigl(\frac{1}{6}\bigl(1+o(1)\bigr)\frac{\log X\log\log\log X}{\log\log X}\Bigr).
$$
This completes the proof of Theorem \ref{thm:1}.

\section{A heuristic argument for a lower bound}

For every prime $p>3$, the polynomial $g(x)$ has exactly two roots modulo $p$ if $p \equiv 1 \pmod{6}$, and no roots otherwise. Moreover, $g(x)$ has no root modulo $2$ and a double root modulo $3$.

To construct an element of $\mathcal{S}$ ``on demand'', fix a prime number $p_{1} \equiv 1 \pmod{6}$. Next, consider the factorization $g(p_{1}) = q_{1}^{\alpha_{1}}\cdots q_{j}^{\alpha_{j}}$, and choose one of its prime factors, say $p_{2} = q_{\ell}$, that is neither too small nor necessarily very large.

Observe that $p_{2} \equiv 1 \pmod{6}$. We then consider $g(p_{2})$ and repeat the same procedure, choosing at each step a new prime factor of $g(p_{i})$, until $p_{1}$ reappears as a prime factor, at which point we stop.

In this way, we obtain a sequence of distinct prime numbers $p_{1},\dots,p_{m}$ such that $p_{i+1}\mid g(p_{i})$ for every $i = 1, \dots, m-1$, and $p_{1}\mid g(p_{m})$. It then follows that
\begin{equation}\label{heur_1}
\prod_{i=1}^{m} p_{i} \in \mathcal{S}.
\end{equation}

The obvious question is the following: ``Can one always find such a prime $p_{m}$, and if so, how large is $m$ in terms of $p_{1}$?'' We are unable to answer this question, but at first sight it seems that the random variable $Z \sim \mathrm{Ber}\bigl(\frac{2}{p_{1}-1}\bigr)$ provides a reasonable model for the probability of success of one ``iteration'' from the second term onward. Indeed, there are two favorable residue classes for $p_{i}$ modulo $p_{1}$ among at most $p_{1}-1$ possibilities, since $p_{i} \not\equiv 0 \pmod{p_{1}}$ by construction.

Assuming that this model is valid, let $Z_{i} \sim \mathrm{Ber}\bigl(\frac{2}{p_{1}-1}\bigr)$, $i=1,2,\dots$, be independent random variables. Then the index of the first success follows a geometric distribution with parameter $\frac{2}{p_{1}-1}$. In particular, for any fixed $c > 0$,
\[
\mathbb{P}(m > cp_{1}) = \Bigl(1-\frac{2}{p_{1}-1}\Bigr)^{\lceil cp_{1} \rceil} \sim e^{-2c},
\]
which shows that the natural scale of the waiting time is $p_{1}$. Consequently, for any prescribed confidence level, one can choose explicit constants $0< c_{1}< c_{2}$ such that, for sufficiently large $p_{1}$, the probability that the first success occurs for some
\[
m\in[c_{1}p_{1}, c_{2}p_{1}]
\]
is at least that confidence level. Let us assume that this is the case.

The term ``iteration'' is admittedly somewhat vague, since at each step one must choose a prime factor of $g(p_{i})$. In making these choices, one would like to avoid trajectories that drift too often toward values below, say, $100\sqrt{p_{1}}$.

It seems quite plausible that, for most sufficiently large primes $p_{1}$, there exist at least $\gg c_{3}^{k}$ admissible sequences of length $k$ consisting only of primes in $[M, M^{100}]$, where $M$ is a large fixed parameter. It is then natural to expect that a positive proportion of these suitably chosen sequences eventually return to $p_{1}$, with some
\[
m \in [c_{1}p_{1}, c_{2}p_{1}].
\]
We expect that most of these sequences give rise, via \eqref{heur_1}, to distinct elements of $\mathcal{S}$.

One would thus obtain at least $\gg e^{c_{4}M}$ elements of $\mathcal{S}$. Moreover, these elements would all be bounded by $M^{c_{5}M}$. Setting $x=M^{c_{5}M}$, we find that this heuristic leads to the lower bound
\[
S(x) \gg \exp\Bigl(c_{6}\frac{\log x}{\log\log x}\Bigr)
\]
for some constant $c_{6} > 0$.

Nevertheless, it is still not known whether $\mathcal{S}$ is infinite. Perhaps one could try to prove that the sequence obtained by repeatedly taking a prime factor left over from a known solution never terminates. For example, $614341 = 7 \cdot 13 \cdot 43 \cdot 157 \in \mathcal{S}$, $h(157) = 7 \cdot 3499$; therefore, $2149579159 = 614341 \cdot 3499 \in \mathcal{S}$ and so on.

{\sc Département de mathématiques et de statistique, Université Laval, Pavillon Alexandre-Vachon, 1045 Avenue de la Médecine, Québec, QC G1V 0A6} \\
{\it E-mail address:} {\tt jmdk@mat.ulaval.ca}

{\sc Département de mathématiques et de statistique, Université Laval, Pavillon Alexandre-Vachon, 1045 Avenue de la Médecine, Québec, QC G1V 0A6} \\
{\it E-mail address:} {\tt Patrick.Letendre.1@ulaval.ca}

\end{document}